\documentclass{article}
\usepackage[utf8]{inputenc}
\usepackage{lipsum}
\usepackage[finnish,swedish,english]{babel}
\usepackage{url}
\usepackage{graphicx}
\usepackage{amsmath}
\usepackage{amsthm}
\usepackage{amssymb}
\usepackage{mathrsfs}
\usepackage{setspace}
\usepackage{color}

\newtheorem{theorem}{Theorem}[section]
\newtheorem{lemma}[theorem]{Lemma}

\newtheorem{corollary}[theorem]{Corollary}

\newtheorem{definition}[theorem]{Definition}

\author{Erik Sj\"oland}
\title{Enumeration of monochromatic three term arithmetic progressions in two-colorings of any finite group}

\begin{document}

\newpage

\maketitle

\begin{abstract}
There are many extremely challenging problems about existence of monochromatic arithmetic progressions in colorings of groups. Many theorems hold only for abelian groups as results on non-abelian groups are often much more difficult to obtain. In this research project we do not only determine existence, but study the more general problem of counting them. We formulate the enumeration problem as a problem in real algebraic geometry and then use state of the art computational methods in semidefinite programming and representation theory to derive lower bounds for the number of monochromatic arithmetic progressions in any finite group.
\end{abstract}

\tableofcontents

\newpage

\section{Introduction}

In Ramsey theory colors or density are commonly used to force structures. It follows from Szemer\'edi's theorem that there exist monochromatic arithmetic progressions of any length if we color the integers with a finite number of colors. Similarly there are arithmetic progressions of any length in any subset of the integers of positive density.

We can explore finite versions of these statements if we replace the integers by $[n]$, or $\mathbb{Z}_n$. In this article we will focus on the cyclic group, as it is easier to analyze because of the symmetries. In most of the literature the posed question has been about how large $n$ has to be for us to find a monochromatic progression of a desired length when $\mathbb{Z}_n$ is colored by a fixed number of colors. A more difficult problem is to determine how many monochromatic progressions there are of desired length when $\mathbb{Z}_n$ is colored by a fixed number of colors given $n$.

Many results on monochromatic progressions are first obtained for cyclic groups and then extended to results for other groups. Many results can be extended to abelian groups using the same machinery as for the cyclic group, whereas results about non-abelian groups are usually very difficult and require a completely different set of tools.

In this paper we reformulate the problem of counting monochromatic arithmetic progressions to a problem in real algebraic geometry that can be attacked by state of the art optimization theory. The methods allows us to find a lower bound to the amount of monochromatic progressions in any finite group, including non-abelian groups. One could find optimal lower bounds for small groups by explicitly counting them for all different colorings, but as this can be done only for very small groups it has not been considered in this article.

In the next section we present our results. In later sections we present our methods and proofs.

\section{Results}
The main theorem of this paper holds for any finite group $G$, including non-abelian groups for which very little is known about arithmetic progressions. The only information that is needed to get a lower bound for a specific group $G$ is the number of elements of the different orders of $G$. The lower bound is sharp for some groups, for example $\mathbb{Z}_p$ with $p$ prime, but is not optimal for most groups. We have further included Table \ref{tab:examples} where we have calculated the lower bound for some small symmetric groups.

\begin{theorem}
\label{thm:groupcase}
Let $G$ be any finite group and let $R(3,G,2)$ denote the minimal number of monochromatic $3$-term arithmetic progressions in any two-coloring of $G$. Let $G_k$ denote the set of elements of $G$ of order $k$, $N=|G|$ and $N_k=|G_k|$. Denote the Euler phi function $\phi(k)=|\{ t \in \{1,\dots,k\}: t  \textrm{ and } k \textrm{ are coprime}\}|$. Let $K=\{k \in \{5,\dots,n\} : \phi(k) \geq \frac{3k}{4}\}$. For any $G$ there are $\sum_{k=4}^n \frac{N\cdot N_k}{2} + \frac{N \cdot N_3}{24}$ arithmetic progressions of length 3. At least
\[
\begin{array}{rl}
R(3,G,2) \geq &\displaystyle \sum_{k \in K}  \frac{N\cdot N_k}{8}(1- 3\frac{k-\phi(k)}{\phi(k)})
\end{array}
\]
of them are monochromatic in a 2-coloring of $G$.
\end{theorem}

\begin{table}
\center
\begin{tabular}{ c | c | c}
Group $G$ & Number of 3-APs & Lower bound for $R(3,G,2)$ \\
\hline
$S_5$ & $4540$ & $90$ \\
$S_6$ & $205440$ & $3240$ \\
$S_7$ & $11307660$ & $306180$ \\
$S_8$ & $774278400$ & $16208640$ \\
\end{tabular}
\caption{Calculation of  $\sum_{k \in K}  \frac{N\cdot N_k}{8}(1- 3\frac{k-\phi(k)}{\phi(k)})$ for small symmetric groups}
\label{tab:examples}
\end{table}

\section{Polynomial optimization}
\label{sec:poly}
To obtain the main theorem we use methods from real algebraic geometry. It is important to note that the proof of the main theorem can be understood without this section, even though the methods played a vital role when finding the sum of squares based certificate in the proof. In this article we only give the elementary definitions relating to polynomial optimization that are needed to prove the results. For a more extensive review of the topic we refer to \cite{Laurent2009}, and for implementation aspects we refer to \cite{Sjoland_methods}.

Given polynomials $f(x),g_1(x),\dots,g_m(x)$ we define a polynomial optimization problem to be a problem on the form
\[
\begin{array}{rll}
\rho_* = \inf  & \displaystyle  f(x) \\
\textnormal{subject to} & \displaystyle  g_1(x) \geq 0, \dots, g_m(x) \geq 0, \\
& \displaystyle  x \in \mathbb{R}^n,
\end{array}
\]
The problem can be reformulated as follows
\[
\begin{array}{rll}
\rho^* = \sup  & \lambda \\
\textnormal{subject to} & \displaystyle f(x) - \lambda \geq 0, g_1(x) \geq 0, \dots, g_m(x) \geq 0 \\
& \displaystyle \lambda \in \mathbb{R}, x \in \mathbb{R}^n
\end{array}
\]
where $f(x), g_1(x),\dots,g_m(x)$ are the same polynomials as before. We refer to the book by Lasserre \cite{Lasserre2010} for an extensive discussion on the relationship between these problems. It is for example easy to see that $\rho_* = \rho^*$.

One of the most challenging problems in real algebraic geometry is to find the most useful relationships between nonnegative polynomials and sums of squares. These relationships are known as Positivstellens\"atze. Let $X$ be a formal indeterminate and let us introduce the following optimization problem:
\[
\begin{array}{rll}
\sigma^* = \sup  & \lambda \\
\textnormal{subject to} & \displaystyle f(X)-\lambda= \sigma_0 + \sum_{i=1}^m \sigma_ig_i \\
& \displaystyle   \sigma_i \textrm{ is a sum of squares.}
\end{array}
\]
One can easily see that $\sigma^* \leq \rho^*$ holds, and under some technical conditions (Archimedean) it was proven by Putinar that $\sigma^* = \rho^*$. The equality holds because of Putinar's Positivstellensatz, which is discussed further in \cite{Sjoland_methods}. The optimization problem can be relaxed by bounding the degrees of all involved monomials to $d$, lets call the solution $\sigma^*_d$. It holds that $\sigma_{d_1}^* \leq \sigma_{d_2}^*$ for $d_1 < d_2$, and it was proven by Lasserre \cite{Lasserre2001} that if the Archimedean condition hold, then
\[
\lim_{d \rightarrow \infty} \sigma^*_d = \sigma^* = \rho^* = \rho_*.
\]
The positivity condition is rewritten as a sum of squares condition because the latter is equivalent to a semidefinite condition: $f(X)$ is a sum of squares of degree $2d$ if and only if $f(X) = v_d^T Q v_d$ for some positive semidefinite matrix $Q$, where $v_d$ is the vector of all monomials up to degree $d$. This makes it possible to use semidefinite programming to find $\sigma^*_d$ as well as a sum of squares based certificate for the lower bound of our original polynomial, $f(X) = \sigma^*_d + \sigma_0 + \sum_{i=1}^m \sigma_i g_i \geq \sigma^*_d$.

\section{Exploiting symmetries in semidefinite programming}
\label{sec:Sym}
This section can be skipped by the reader who is just interested in the final proof of this article. The methods in this section were used and implemented to find parts of the numerical results that lead to the final proof, and is thus intended for the reader who wants to understand the full process from problem formulation to the end certificate, or solve similar problems with symmetries.

Let  $C$ and $A_1, \dots, A_m$ be real symmetric matrices and let $b_1,\dots,b_m$ be real numbers. To reduce the order of the matrices in the semidefinite programming problem
\[ \max \{\mathrm{tr}(CX) ~ | ~X \textrm{ positive semidefinite},  \mathrm{tr}(A_i X)=b_i \textrm{ for } i = 1, \dots, m\} \]
when it is invariant under all the actions of a group is the goal of this section.

As in \cite{KPSchrijver2007} and \cite{Klerk2011}, we use a $\ast$--representation to reduce the dimension of the problem. For the reader interested in $\ast$--algebras we recommend the book by Takesaki \cite{Takesaki2002}. A collection of several new methods, including the one we use, to solve invariant semidefinite programs can be found in \cite{Bachoc2012}. Other important recent contributions in this area include \cite{Kanno2001, Gatermann2004, Vallentin2009, Maehara2010, Murota2010, Riener2013}.

\begin{definition}
A \emph{matrix $\ast$-algebra} is a collection of matrices that is closed under addition, scalar multiplication, matrix multiplication and transposition.
\end{definition}

Let $G$ be a finite group that acts on a finite set $Z$. Define a homomorphism $h : G \rightarrow S_{Z}$, where $S_{Z}$ is the group of all permutations of $Z$. For every $g \in G$ there is a permutation $h_g=h(g)$ of the elements of $Z$ with the properties $h_{g g'}=h_{g}h_{g'}$ and $h_{g^{-1}}=h_{g}^{-1}$. For every permutation $h_{g}$, there is a corresponding permutation matrix $M_{g} \in \{0,1\}^{Z \times Z}$, element-wise defined by
\[
(M_{g})_{i,j}=
\left\{
\begin{array}{rl}
 1 & \textrm{ if } h_{g}(i )= j, \\ 
 0 & \textrm{otherwise} 
\end{array}
\right.
\]
for all $i,j \in Z$.
Let the span of these permutation matrices define the matrix $\ast$-algebra
\[
\mathcal{A} = \left\{ \sum_{g \in G} \lambda_g M_{g} ~ | ~ \lambda_g \in \mathbb{R} \right\}.
\]
The matrices $X$ satisfying $XM_{g}=M_{g}X$ for all $g \in G$ are the \emph{matrices invariant under the action of $G$}. The $\ast$-algebra consisting of the collection of all such matrices, 
\[
\mathcal{A'} = \{X \in \mathbb{R}^{n \times n} | XM=MX \textrm{ for all } M \in \mathcal{A} \},
\]
is known as the \emph{commutant} of $\mathcal{A}$. We let $d=\dim\mathcal{A'}$ denote the dimension of the commutant.

The commutant has a basis of $\{0,1\}$-matrices, which we denote $E_1,\dots,E_d$, with the property that $\sum_{i=1}^d E_i = J$, where $J$ is the all-one $Z \times Z$-matrix.

We form a new basis for the commutant by normalizing the $E_i$s
\[
B_i = \frac{1}{\sqrt{tr(E_i^TE_i)}}E_i.
\]
The new basis has the property that  $\mathrm{tr}(B_i^TB_j) = \delta_{i,j}$, where $\delta_{i,j}$ is the Kronecker delta.

From the new basis we introduce \emph{multiplication parameters} $\lambda_{i,j}^k$ by 
\[B_iB_j = \sum_{k=1}^d \lambda_{i,j}^kB_k \]
for $i,j,k = 1,\dots,d$.

We introduce a new set of matrices $L_1,\dots,L_d$ by
\[
(L_k)_{i,j} = \lambda_{k,j}^i 
\]
for $k,i,j =1,\dots,d$. The matrices $L_1,\dots,L_d$ are $d \times d$ matrices that span the linear space
\[ 
\mathcal{L}=\{\sum_{i=1}^d x_iL_i : x_1,\dots,x_d \in \mathbb{R} \}.
\]

\begin{theorem}[\cite{KPSchrijver2007}]
The linear function $\phi : \mathcal{A'} \rightarrow \mathbb{R}^{d \times d}$ defined by $\phi(B_i) = L_i$ for $i=1,\dots,d$ is a bijection, which additionally satisfies $\phi(XY)=\phi(X)\phi(Y)$ and $\phi(X^T)=\phi(X)^T$ for all $X,Y \in \mathcal{A}'$.
\end{theorem}

\begin{corollary}[\cite{KPSchrijver2007}]
$\sum_{i=1}^d x_iB_i$ is positive semidefinite if and only if $\sum_{i=1}^d x_iL_i$ is positive semidefinite.
\label{cor:schrijver}
\end{corollary}

If it is possible to find a solution $X \in \mathcal{A}'$, then Corollary \ref{cor:schrijver} can be used to reduce the size of the matrix in the linear matrix inequality. Let us show that this is possible:
\begin{lemma}
\label{lem:groupaverage}
There is a solution $X \in \mathcal{A}'$ to a $G$-invariant semidefinite program
\[ 
\max \{\mathrm{tr}(CX) ~ | ~X \textrm{ positive semidefinite},  \mathrm{tr}(A_i X)=b_i \textrm{ for } i = 1, \dots, m\}. 
\]
\end{lemma}

\begin{proof}

Let $C,A_1,\dots,A_m$ be $Z \times Z$ matrices commuting with $M_g$ for all $g \in G$. If $X$ is an optimal solution to the optimization problem then the group average, $X'=\frac{1}{|G|} \sum_{g \in G}M_g X M_g^T$, is also an optimal solution:
It is feasible since 
\[
\begin{array}{rl}
\displaystyle \mathrm{tr}(A_jX') &= \displaystyle  \mathrm{tr}(A_j\frac{1}{|G|} \sum_{g \in G}M_g X M_g^T) \\
&= \displaystyle  \mathrm{tr}(\frac{1}{|G|} \sum_{g \in G}M_g A_jX M_g^T) \\
&= \displaystyle  \mathrm{tr}(\frac{1}{|G|} \sum_{g \in G}A_jX ) \\
&= \displaystyle  \mathrm{tr}(A_jX),
\end{array}
\]
where we have used that the well-known fact that the trace is invariant under change of basis. By the same argument  $\mathrm{tr}(CX') =\mathrm{tr}(CX)$, which implies that $X'$ is optimal. It is easy to see that $X' \in \mathcal{A}'$.
\end{proof}

The following theorem follows, which gives a tremendous computational advantage when $d$ is significantly smaller than $|Z|$:

\begin{theorem}[\cite{KPSchrijver2007}]
\label{thm:Schrijver}
The $G$-invariant semidefinite program
\[
\max \{\mathrm{tr}(CX) ~ | ~X \succeq 0,  \mathrm{tr}(A_i X)=b_i \textrm{ for } i = 1, \dots, m\} 
\]
has a solution $X = \sum_{i=1}^d x_iB_i$ that can be obtained by
\[
\max \{\mathrm{tr}(CX) ~ | ~ \sum_{i=1}^d x_iL_i \succeq 0, \mathrm{tr}(A_i \sum_{j=1}^d B_jx_j)=b_i \textrm{ for } i = 1, \dots, m\}.
\]
\end{theorem}

\section{Problem formulated as a semidefinite program}
\label{sec:methodscolor}
We follow the convention that an arithmetic progression in $G$ of length $k$ is a set of $k$ distinct element $\{a,b \cdot a,\dots,b^{k-1} \cdot a \}$ where $a \in G$, $b \in \mathbb{Z}^+$. Also, as $\{1,2,3\}$, $\{1,3,2\}$, $\{2,1,3\}$, $\{2,3,1\}$, $\{3,1,2\}$ and $\{3,2,1\}$ denote the same set they are considered as the same arithmetic progression, thus when summing over all arithmetic progressions the set is only considered once.

Let $\chi : G \rightarrow \{-1,1\}$ be a $2$-coloring of the finite group $G$, and for simplicity let $x_g = \chi(g)$ for all $g \in G$. Furthermore, let $x_{_G}$ denote the vector of all variables $x_g$. Let us introduce the polynomial
\[
\begin{array}{rl}
p(x_a,x_b,x_c) &= \displaystyle \frac{(x_a+1)(x_b+1)(x_c+1)-(x_a-1)(x_b-1)(x_c-1)}{8}  \\
&= \displaystyle \frac{x_ax_b+x_ax_c+x_bx_c+1}{4},
\end{array}
\]
where $a,b,c \in G$, which has the property that
\[
p(x_a,x_b,x_c) = \left\{ 
\begin{array}{ll}
1 & \text{if }  x_a=x_b=x_c\\
0 & \text{otherwise.}
\end{array} \right. 
\]
The polynomial $p$ is one when $a,b,c$ are the same color and zero otherwise. It follows that
\[
R(3,G,2)=  \min_{x_{_G} \in \{-1,1\}^{|G|}} \displaystyle\sum_{\{a,b,c\} \textrm{ is an A.P. in } G}p(x_a,x_b,x_c).
\]

The integer problem is relaxed to a problem on the hypercube to obtain a lower bound for $R(3,G,2)$ . Since any solution of the integer program is also a solution to the hypercube problem we have
\[
\begin{array}{rl}
R(3,G,2) &\displaystyle \geq \min_{ x_{_G} \in [-1,1]^{|G|}} \sum_{\{a,b,c\} \textrm{ is an A.P. in } G}p(x_a,x_b,x_c) \\
& = \displaystyle \min_{ x_{_G} \in [-1,1]^{|G|}} \sum_{\{a,b,c\} \textrm{ is an A.P. in } G} \frac{x_ax_b+x_ax_c+x_bx_c+1}{4} \\
& = \displaystyle \min_{ x_{_G} \in [-1,1]^{|G|}} \frac{p_{_G}}{4}+\sum_{\{a,b,c\} \textrm{ is an A.P. in } G}  \frac{1}{4}
\end{array}
\]
where 
\[ p_{_G} = \sum_{\{a,b,c\} \textrm{ is an A.P. in } G}x_ax_b+x_ax_c+x_bx_c. \]

We immediately get a lower bound for $R(3,G,2)$ by finding a lower bound for the homogeneous degree 2 polynomial $p_{_G}$. The coefficient of $x_ax_b$ in $p_{_G}$ equals the number of times the pair $(a,b)$ occurs in a 3-arithmetic progression, which depends on the group $G$. After the coefficients are found the state-of-the-art methods surveyed in Sections \ref{sec:poly} and \ref{sec:Sym} are used to find a lower bound for 
\[
 \min_{ x_{_G} \in [-1,1]^{|G|}}  p_{_G}.
\]

Let us use the degree 3 relaxation of Putinar's Positivstellensatz, and let $\lambda^*$ denote the maximal lower bound using this relaxation. Denote the elements of $G$ by $g_1,\dots,g_{|G|}$ and let $v$ be the vector of all monomials of degree less or equal than one in the formal indeterminates $X_{g_1},\dots,X_{g_{|G|}}$; $v=[1, X_{g_1}, \dots, X_{g_{|G|}}]^T$. We get 
\[
\begin{array}{rl}
\lambda^* = \max  & \lambda \\
\textnormal{subject to:} & \displaystyle p_{_G} - \lambda = v^T Q_0 v + \sum_{g \in G} v^T Q_g^+ v(1+X_g) + \sum_{g \in G}  v^T Q_g^- v (1-X_g) \\
& \lambda \in \mathbb{R} \\
& \displaystyle Q_0,Q_g^+,Q_g^- \succeq 0 \textnormal{ for all } g\in G.
\end{array}
\]

For simplicity, let the $Q$-matrices entries be indexed by the set $\{1$, $g_1$, $\dots$, $g_{|G|}\}$. Let $A \rtimes B$ denote the semidirect product of $A$ and $B$, defined such that $(a,b) \in A \rtimes B$ takes $i$ to $a+bi$ (here $+$ denotes the action of the group $G$, and $bi$ denotes the repeated action of $i$ on itself $b$ times: $bi=\sum_{j=1}^b i$). Arithmetic progressions are invariant under affine transformations: if $(a,b) \in G \rtimes \mathbb{Z}^+$, and $\{i,j,k\}$ is an arithmetic progression, then $(a,b) \cdot \{i,j,k\}=\{a+bi,a+bj,a+bk\}$ is also an arithmetic progressions. It follows that the semidefinite program is invariant under affine transformations, and thus that we can find an invariant solution by Lemma \ref{lem:groupaverage}. This implies that the degree 3 relaxation has a solution for which $Q_0(g_i,g_j)=Q_0(a + b  g_i, a + b  g_j)$ and $Q_0(1,g_i) = Q_0(g_i,1) =Q_0(1,a+b  g_i)$  for all $(a,b) \in G \rtimes \mathbb{Z}^+$ and $g_i,g_j \in G$. Also $Q_{g_i}^+(g_j,g_k) = Q_{a + b  g_i}^+(a + b  g_j, a + b  g_k)$, $Q_{g_i}^+(1,g_k) = Q_{g_i}^+(g_k,1) = Q_{a + b  g_i}^+(1, a + b  g_k)$ and $Q_{g_i}^+(1,1)  = Q_{a + b  g_i}^+(1, 1)$. By the same argument we get similar equalities for the indices of the matrices $Q_{g_i}^-$. From the equalities It is easy to see that the matrices $Q_{g_i}^+$ and $Q_{g_i}^-$ are obtained by simultaneously permuting the rows and columns of $Q_{g_j}^+$ and $Q_{g_j}^-$ respectively, and hence it is enough to require that $Q_{g_1}^+$ and $Q_{g_1}^-$ are positive semidefinite where $g_1 \in G$ is the identity element. 

In conclusion we see that the number of variables can be reduced significantly, and that only three $|G|+1 \times |G|+1$-matrices are required to be positive semidefinite instead of the $2|G|+1$ matrices required in the original formulation. The dimension of the commutant is small for some groups, and in those cases the size of these matrices can be reduced further using Theorem \ref{thm:Schrijver}.

The certificates for the lower bounds we get from these methods are numerical, and to make them algebraic additional analysis of the numerical data, and possibly further restrictions, must be done. There is no general way to find algebraic certificates, and for many problem it might not even be possible. In this paper it is a vital step to find as the numerical certificates only gives certificates for one group at the time whereas we need a certificate for all groups.

It is in general very difficult to go from numerical patterns to algebraic certificates. For a specific group all information required to find a lower bound can be found in an eigenvalue decomposition of the involved matrices, but there is no general way of finding the optimal algebraic lower bound when the different eigenvalues and eigenvectors have decimal expansions that cannot trivially be translated into algebraic numbers. If one is interested in a rational approximation to the lower bound one can use methods by Parrilo and Peyrl \cite{Parrilo2008}. These methods gives a good certificate for a specific group but does not help when one want to find certificates for an infinite family of groups.

To find a certificate for all groups, one of the tricks we used was to restrict the SDP above by requiring that some further entires equal one another, in order to at least get a lower bound for $\lambda^*$. There are also many other ways to restrict the SDP further, including setting elements to zero. Restricting the SDP either decreases the optimal value or reduces the number of solutions to the original problem. In the ideal scenario one can keep restricting the SDP until there is only one solution, without any change in the optimal value.

\section{Proof of Theorem \ref{thm:groupcase}}

To make the computations in the proofs that follow readable we introduce additional notation:
\[
\sigma(a;b_0,b_1,\ldots,b_{n-1})=a+ \sum_{i,j\in \mathbb{Z}_n} b_{j-i} X_iX_j.
\]

By elementary calculations we have the following equalities, which are needed in the proofs:

\[
I_1=\displaystyle  \sum_{i\in \mathbb{Z}_{n}} (1-X_i^2)  =  \sigma(n;  -1,  0, \dots ,  0),
\]
\[
I_2=\displaystyle \Big( \sum_{i\in \mathbb{Z}_{n}} X_i \Big)^2  = \sigma\Big( 0;  1,  2,  \dots,  2 \Big), 
\]
\[
I_3^j=\displaystyle \frac{1}{2}\sum_{i \in \mathbb{Z}_{n}}(X_i-X_{i+j})^2 = \sigma\Big( 0; 1 , 0 , \dots, 0 ,-1 , 0, \dots, 0,-1,0,\dots,0 \Big),
\]

The first one is non-negative since we for the problem require that $-1 \leq x_i \leq 1$ and the other polynomials are non-negative since they are sums of squares.

\begin{proof}[Proof of Theorem \ref{thm:groupcase}]
Let us count all arithmetic progressions $\{a,b \cdot a,b \cdot b \cdot a\}$, with $(a,b) \in G \times G_k$. There is a cyclic subgroup of $G$ with elements $\{1,b,b^2,\dots,b^{k-1}\}$. Let us write $k=p_1^{e_1}p_2^{e_2} \cdots p_s^{e_s}$ for $1<p_1<\dots<p_s$ distinct primes and $e_1,\dots,e_s$ positive integers. The elements $\{b^{p_i},b^{2p_i},\dots,b^{(\frac{k}{p_i}-1)p_i}\}$ are of order less  than $k$ for all $i$ whereas the elements $U=\{b^t :  t$ and $k$ are coprime$\}$ are of order $k$. Note that $\phi(k)=|U|$. If $k<3$ there are no arithmetic progressions, so in all the following calculations we will always assume that $k \geq 3$. If $3 \nmid k$, then all triples $\{(1,b^i,b^{2i}) : b^i \in U \}$ will be distinct arithmetic progressions. Since there are $\frac{N_k}{\phi(k)}$ different cyclic subgroups of $G$ of this type there are $\phi(k)\frac{N_k}{\phi(k)}=N_k$ arithmetic progressions of the form $\{1,b^i,b^{2i}\}$, where $b^i$ is an element of order $k$. Since $\{a, b\cdot a, b\cdot b\cdot a\}$ is an arithmetic progression if and only if $\{1,b,b\cdot b\}$ is an arithmetic progression it follows that there are $\frac{N\cdot N_k}{2}$ arithmetic progressions with $(a,b) \in G \times G_k $ if $3 \nmid k$. If $3 | k$ we have to be careful so that we do not count arithmetic progressions of the form $\{a,b^{\frac{k}{3}}\cdot a,b^{\frac{2k}{3}}\cdot a\}$ three times. Since $b^{\frac{k}{3}}$ will be of lower order than $k$ if $k>3$ it follows that  this kind of arithmetic progressions only occur for $(a,b) \in G \times G_3$. We conclude that there are $\frac{N\cdot N_k}{2}$ arithmetic progressions with $(a,b) \in G \times G_k $ if $k>3$, and $\frac{N\cdot N_3}{6}$ if $(a,b) \in G \times G_3$.

The number of monochromatic arithmetic progressions is given by
\[
\begin{array}{rl}
R(3,G,2) &=  \displaystyle\sum_{\{a,b,c\} \textrm{ is an A.P. in } G}p(x_a,x_b,x_c) \\
&= \displaystyle\sum_{\{a,b,c\} \textrm{ is an A.P. in } G}\frac{x_ax_b+x_ax_c+x_bx_c+1}{4}.
\end{array}
\]
Let us rewrite this as
\[
R(3,G,2)=\frac{\sum_{k=4}^n \frac{N\cdot N_k}{2} +p_k }{4} + \frac{N \cdot N_3 }{24} + \frac{p_3}{4},
\]
where $p_k$ is the sum of all polynomials $x_sx_{t \cdot s}+x_sx_{t \cdot t \cdot s}+x_{t \cdot b}x_{t \cdot t \cdot s}$ where $t$ is of order $k$. Let us also define the further reduced polynomial $p_k^{(a,b)}$ for a fixed pair $(a,b) \in G \times G_k$ by
\[
p_k^{(a,b)}=\sum_{0 \leq i < j \leq k-1}c_{b^i \cdot a,b^j \cdot a}x_{b^i \cdot a}x_{b^j \cdot a}
\]
where $c_{b^i \cdot a,b^j\cdot a}$ denotes how many times the pair $(b^i\cdot a,b^j\cdot a)$ is in an arithmetic progression $\{s,t \cdot s,t \cdot t \cdot s\}$ with $t$ of order $k$. When $a=1$ we will have a set $\{b_{i_1},\dots,b_{i_{N_k/\phi(k)}} \}$ of elements of order $k$ such that $b_{i_j}^{c_1} \neq b_{i_k}^{c_2}$ for all $j,k \in \{1,\dots,N_k/\phi(k) \}$ and all $c_1,c_2 \in \{1,\dots,k-1\}$, and can write
\[
p_k= \frac{1}{k} \sum_{a\in G} \sum_{j=1}^{N_k/\phi(k)} p_k^{(a,b_{i_j})}.
\]
It follows that if $p_k^{(1,b)} \geq c$ for a $b$ of order $k$, then $p_k \geq \frac{N}{k}\frac{N_k}{\phi(k)}c $.

{\bf Let k mod 2 = 1, k mod 3 $\neq$ 0:}
It is fairly easy to see that if $b^i$ is of order $k$, then so is $b^{-i}$ and $b^{2i}$. Also, if $2 | i$ then $b^{i/2}$ is of order $k$, if $2 \nmid i$ then $b^{(n+i)/2}$ is of order $k$. Hence any $b^i$ is in both arithmetic progressions $b^{-i},1,b^i$ and $1,b^i,b^{2i}$, and also in either $1,b^{i/2},b^i$ or $1,b^{(n+i)/2},b^i$. These are all arithmetic progressions containing both $1$ and $b^i$. Note that the elements that are not of order $k$ (apart from the identity element) are not in any of these arithmetic progressions, hence if $k=p_1^{e_1}\dots p_s^{e_s}$, then
\[
p_k^{(1,b)}(X)=\sum_{i,j\in \mathbb{Z}_k} b_{j-i} X_iX_j=\sigma(0;b_0,b_1,\ldots,b_{k-1})
\]
where $b_{tp_i}=0$ for $t\in \{0,\dots,k-1\}$ and $i\in \{1,\dots,s\}$, and $b_t=3$ for all other $t$. In particular we get
\[
p_k^{(1,b)}(X) = \frac{3}{2}I_2+\frac{3}{2}\sum_{t,i}I_3^{tp_i} +\frac{3}{2}(1+k-\phi(k)-1)I_1 - \frac{3}{2}(k-\phi(k))k \geq - \frac{3}{2}(k-\phi(k))k,
\]
and it follows that 
\[
p_k \geq - \frac{N}{k}\frac{N_k}{\phi(k)}  \frac{3}{2}(k-\phi(k))k=-  \frac{3}{2} \frac{N \cdot N_k}{\phi(k)} (k-\phi(k)),
\]
and furthermore  that
\[
\frac{N\cdot N_k}{2} +p_k \geq \frac{N\cdot N_k}{2}(1- 3\frac{k-\phi(k)}{\phi(k)}).
\]

In cases when $1- 3\frac{k-\phi(k)}{\phi(k)} < 0$, that is when $\phi(k) <\frac{3k}{4}$, we will instead use the trivial bound
\[
\frac{N\cdot N_k}{2} +p_k \geq 0
\]

{\bf k mod 2 = 0: }
$b \in G_k$, and let us color all elements $\{1,b^2,b^4,\dots,b^{k-2}\}$ blue and $\{b,b^3,b^5,\dots,b^{k-1}\}$ red. In an arithmetic progression $\{a,c\cdot a,c\cdot c\cdot a \}$ with $(a,c) \in \{0,b,b^2,\dots,b^{k-1} \} \times \{0,b,b^2,\dots,b^{k-1} \} \cap G_k$ it holds that $a$ and $c\cdot a$ are of different colors. The coloring can be extended too all pairs $(a,c) \in G \times G_k$, and thus there is a coloring without monochromatic arithmetic progressions. Hence we cannot hope to do better than the trivial bound using these methods:
\[
\frac{N\cdot N_k}{2} +p_k \geq 0.
\]
 
{\bf k mod 3 = 0: }
Let us color the elements $\{0,b,b^3,b^4,\dots,b^{k-3},b^{k-2}\}$ blue and $\{b^2,b^5,b^8,\dots,b^{k-1}\}$ red. As when k mod 2 = 0 we consider arithmetic progressions $\{a,c\cdot a,c\cdot c\cdot a\}$ with $(a,c) \in \{0,b,b^2,\dots,b^{k-1} \} \times \{0,b,b^2,\dots,b^{k-1} \} \cap G_k$. It is easy to see that either $a$ and $c\cdot a$ or $a$ and $c \cdot c \cdot a$ are of different colors. Again there is a coloring without monochromatic arithmetic progressions and so we cannot do better than the trivial bound:
\[
\frac{N\cdot N_k}{2} +p_k \geq 0
\]
for $k>3$ and
\[
\frac{N\cdot N_3}{6} +p_3 \geq 0
\]
for $k=3$.

Let $K=\{k \in \{5,\dots,n\} : \phi(k) \geq \frac{3k}{4}\}$. If $2$ or $3$ divide $k$ then $\phi(k) < \frac{3k}{4}$, and hence none of those numbers are included in $K$. Summing up all cases we get
\[
R(3,G,2) \geq  \frac{\sum_{k \in K}  \frac{N\cdot N_k}{2}(1- 3\frac{k-\phi(k)}{\phi(k)})}{4}.
\]
\end{proof}

\section{Methods for longer arithmetic progressions}
Let $\chi : G \rightarrow \{-1,1\}$ be a $2$-coloring of the group $G$, and let $x_g = \chi(g)$ for all $g \in G$. Let also $x_{_G}$ denote the vector of all variables $x_g$. For $a,b,c \in G$, let us introduce the polynomial
\[
p(x_{a_1},\dots,x_{a_k}) = \displaystyle \frac{(1+x_{a_1})\cdots(1+x_{a_k})+(1-x_{a_1})\cdots(1-x_{a_k})}{2^k}, 
\]
which has the property that
\[
p(x_{a_1},\dots,x_{a_k}) = \left\{ 
\begin{array}{ll}
1 & \text{if }  x_{a_1}=\dots=x_{a_k}\\
0 & \text{otherwise.}
\end{array} \right. 
\]
In other words, the polynomial $p$ is one when $\{a_1,\dots,a_k\}$ is a monochromatic arithmetic progression and zero otherwise. It follows that
\[
R(k,G,2)=  \min_{x_{_G} \in \{-1,1\}^{|G|}} \displaystyle\sum_{\{a_1,\dots,a_k\} \textrm{ is an A.P. in } G}p(x_{a_1},\dots,x_{a_k}).
\]

To find a lower bound for $R(k,G,2)$ we relax the integer quadratic optimization problem to an optimization problem on the hypercube. Since it is a relaxation, i.e. any solution of the integer program is also a solution to the hypercube problem, we have
\[
R(k,G,2) \geq \min_{ x_{_G} \in [-1,1]^{|G|}} \displaystyle\sum_{\{a_1,\dots,a_k\} \textrm{ is an A.P. in } G}p(x_{a_1},\dots,x_{a_k}),
\]
an optimization problem that we can find lower bounds for using Putinar's Positivstellensatz and the Lasserre Hierarchy.

\section{Related problems}
The methods developed in this article may also be applied in a wide variety of similar problems. 

Let $R(3,[n],2)$ denote the minimal number of monochromatic arithmetic progressions of length $3$ in a $2$-coloring of $[n]$. Asymptotic bounds for $R(3,[n],2)$ have been found for large $n$ using other methods \cite{Parrilo2008_2}:
\[ \frac{1675}{32768}n^2(1+o(1)) \leq R(3,[n],2) \leq \frac{117}{2192}n^2(1+o(1)). \]
The author has in collaboration with Oscar Kivinen found numerical results suggesting that the lower bound can be improved to $0.052341 \dots$ using a degree $3$-relaxation ($\frac{1675}{32768} = 0.05111 \dots$ and $\frac{117}{2192} =0.05337 \dots$) . The major challenge is to use the numerical results to obtain an algebraic certificate because of the lack of symmetries in the problem. For large $n$ the numerical information seems to converge towards a fractal, and to obtain the improved lower bound based on the proposed methods one would need to fully understand the behavior of the fractal. Even if it is not possible to understand the fractal, one can probably do approximations that would improve on the lower bound. This is work under progress.

Let $R(3,\mathbb{Z}_n,2)$ denote the minimal number of monochromatic arithmetic progressions of length $3$ in a $2$-coloring of the cyclic group $\mathbb{Z}_n$. Optimal, or a constant from optimal, lower bounds for $R(3,\mathbb{Z}_n,2)$ have been found for all $n$ \cite{Sjoland_cyclic}:
\[
n^2/8-c_1n+c_2 \leq R(3,\mathbb{Z}_n,2) \leq n^2/8-c_1n+c_3,
\] 
where the constants depends on the modular arithmetic and are tabulated in the following table.
\[
\begin{array}{c|c|c|c}
n \mod 24 & c_1 & c_2 & c_3 \\
\hline
1,5,7,11,13,17,19,23 & 1/2 & 3/8 & 3/8 \\
8,16 & 1 & 0 & 0 \\
2,10 & 1 & 3/2 & 3/2 \\
4,20 & 1 & 0 & 2 \\
14,22  & 1 & 3/2 & 3/2 \\
3,9,15,21 &  7/6 & 3/8 & 27/8 \\ 
0 &  5/3 & 0  & 0 \\ 
12 &  5/3 & 0 & 18 \\ 
6,18 &  5/3 &1/2 & 27/2 \\ 
\end{array}
\]
A corollary is that we can find an optimal, or a constant from optimal, lower bound for the number of monochromatic arithmetic progressions for the dihedral group $D_{2n}$ for any $n$:
\[
R(3,D_{2n},2)=2R(3,\mathbb{Z}_n,2).
\]
In particular 
\[
n^2/4-2c_1n+2c_2 \leq R(D_{2n};3) \leq n^2/4-2c_1n+2c_3
\]
where the constants can be found in the table above.

Let $R(4,\mathbb{Z}_n,2)$ denote the minimal number of monochromatic arithmetic progressions of length $4$ in a $2$-coloring of the cyclic group $\mathbb{Z}_n$. Asymptotic bounds for $R(4,\mathbb{Z}_n,2)$ have been found for large $n$ in \cite{Wolf2010}, and the bounds have since then been improved in \cite[Theorem 1.1, 1.2, 1.3]{Lu2012} to:
\[
\frac{7}{192}p^2(1+o(1)) \leq R(4,\mathbb{Z}_p,2) \leq \frac{17}{300}p^2(1+o(1))
\]
when $p$ is prime, and for other $n$
\[
c_1n^2(1+o(1)) \leq R(4,\mathbb{Z}_n,2) \leq c_2n^2(1+o(1))
\]
where the constants depends on the modular arithmetic on $n$ in accordance with the following table
\[
\begin{array}{c|c|c}
n \mod 4 & c_1 & c_2 \\
\hline
1,3 & 7/192 & 17/300\\
0 & 2/66 & 8543/1452000 \\
2 & 7/192 & 8543/1452000 \\
\end{array}
\]
Furthermore \cite[Theorem 1.5]{Lu2012}
\[
\underline{\lim}_{n \rightarrow \infty} R(4,\mathbb{Z}_n,2) \leq \frac{1}{24},
\]
and it is conjectured that \cite[Conjecture 1.1]{Lu2012}:
\[
\inf_n\{R(4,\mathbb{Z}_n,2) \} = \frac{1}{24}.
\]
The author has tried to improve the bounds using a degree 4-relaxations with further restriction to simplify the problem, but as the author obtained algebraic bounds that are far from the current best bounds it seems like one needs to omit the restrictions and possibly use a higher degree relaxation to get relevant bounds. One of the main challenges is that the solution depends heavily on the modular arithmetic of $n$, and so one needs to numerically find solutions for fairly high values of $n$, which is not feasible for a high degree relaxation. Doing a full degree 4-relaxation is possible but tedious, and since one can only get numerical results for fairly small $n$ it is difficult to find general patterns. One cannot exclude the possibility that this would be enough to improve the bounds, but what is more likely is that it would require additional tricks or techniques after one has found numerical certificates for small $n$.

Let $R(4,[n],2)$ denote the minimal number of monochromatic arithmetic progressions of length $4$ in a 2-coloring of $[n]$. Let furthermore $R(5,\mathbb{Z}_n,2)$ and $R(5,[n],2)$ denote the minimal number of monochromatic arithmetic progressions of length $5$ in $\mathbb{Z}_n$ and $[n]$ respectively. Upper bounds have been found for $R(4,[n],2)$, $R(5,\mathbb{Z}_n,2)$ and $R(5,[n],2)$.
For $n$ large enough \cite[Equation (12)]{Lu2012}:
\[ R(4,[n],2) \leq \frac{1}{72}n^2(1+o(1)). \]
When $n$ is odd and large enough we have \cite[Theorem 1.4]{Lu2012}:
\[
R(5,\mathbb{Z}_n,2) \leq \frac{3629}{131424}n^2(1+o(1)).
\]
When $n$ is even and large enough we have \cite[Theorem 1.4]{Lu2012}:
\[
R(5,\mathbb{Z}_n,2) \leq \frac{3647}{131424}n^2(1+o(1)).
\]
Furthermore \cite[Theorem 1.5]{Lu2012}:
\[
\underline{\lim}_{n \rightarrow \infty} R(5,\mathbb{Z}_n,2) \leq \frac{1}{72}.
\]
Finally when $n$ is large enough we have \cite[Equation (13)]{Lu2012}:
\[ R(5,[n],2) \leq \frac{1}{304}n^2(1+o(1)). \]
These are all upper bounds that have been obtained through good colorings. Lower bounds to all these problems can possibly be found using the methods developed in this article using a relaxation of high enough order. As it is numerically very difficult to find solutions to relaxations of high orders, one should be aware that this is not guaranteed to work in practice.

Another type of problems that one can use the methods developed in this article to solve are enumeration problems in fixed density sets. As in this article it is of interest to count arithmetic progressions. Let $W(k,S,\delta)$ denote the minimal number of arithmetic progressions of length $k$ in any subset of $S$ of cardinality $|S|\delta$. One can for example let $S$ be a group or $[n]$. Note that if one could find strict lower bounds for $W(k,[n],\delta)$ for all $n$, $k$ and $\delta$ this would imply optimal quantitative bounds for Szemer\'edi's theorem. Although it might be too ambitious to try to find strict lower bounds it might still be possible to find bounds good enough to generalize Szemer\'edi's theorem. This kind of results can be obtained by methods very similar to the ones used in this article as discussed in \cite{Sjoland_density}.

\section*{Acknowledgements}
I would like to thank Alexander Engstr\"om for our long discussions on this topic and for all his valuable feedback. I would also want to thank Markus Schweighofer and Cynthia Vinzant for their comments, which have been incorporated in the article.

\end{document}